\documentclass[12pt,reqno]{amsart}

\usepackage{amsmath,amsthm,amscd,amsfonts,amssymb}
\usepackage[utf8]{inputenc}
\usepackage{graphicx}
\usepackage[colorlinks=true, allcolors=blue]{hyperref}
\usepackage[all]{xypic}
\usepackage{enumerate}

\numberwithin{equation}{subsection}

\newcommand{\sqsp}{\renewcommand{\baselinestretch}{1.1}\tiny\normalsize}

\newtheorem{thm}[subsection]{Theorem}
\newtheorem{lem}[subsection]{Lemma}
\newtheorem{prop}[subsection]{Proposition}
\newtheorem{cor}[subsection]{Corollary}
\theoremstyle{definition}
\newtheorem{defn}[subsection]{Definition}
\newtheorem{exs}[subsection]{Examples}
\newtheorem{rem}[subsection]{Remark}

\newcommand{\m}{\mathfrak{m}}
\newcommand{\g}{\mathfrak{g}}
\newcommand{\s}{\mathfrak{s}}
\newcommand{\so}{\mathfrak{so}}
\newcommand{\HS}{\mathrm{HS}}

\newcommand{\SSH}{\mathcal{SS}}
\newcommand{\RS}{\mathcal{RS}}
\newcommand{\Ss}{\mathcal{SS}^*}
\newcommand{\PS}{\mathcal{PS}}
\newcommand{\MS}{\mathcal{MS}}
\newcommand{\A}{\mathcal{A}}
\newcommand{\F}{\mathbb{F}}

\newcommand{\C}{\mathbb{C}}
\newcommand{\ch}{\mathrm{char}(\F)}
\newcommand{\brk}{[\cdot,\cdot]}
\DeclareMathOperator{\End}{End}
\DeclareMathOperator{\Aut}{Aut}

\DeclareMathOperator{\Ker}{Ker}

\DeclareMathOperator{\im}{Im}
\DeclareMathOperator{\ad}{ad}
\DeclareMathOperator{\Span}{span}
\newcommand{\p}{\mathit{\Pi}}
\newcommand{\tp}{\widetilde{\mathit{\Pi}}}

\newcommand*{\cycl}{\mathop{\kern0.9ex{{+}\kern-2.77ex\raise-.6ex%
      \hbox{\huge\hbox{$\circlearrowleft$}}}}\limits}

\title{Classification problem of simple Hom-Lie algebras}
\author{Youness El Kharraf}                 

\begin{document}

\begin{abstract}
First, we construct some families of nonsolvable anticommutative algebras, solvable Lie algebras and even nilpotent Lie algebras, that can be endowed with the structure of a simple Hom-Lie algebra. This situation shows that a classification of simple Hom-Lie algebras would be unrealistic without any further restrictions. We introduce the class of \emph{strongly simple Hom-Lie algebras}, as the class of anticommutative algebras that are simple Hom-Lie with respect to all their twisting maps. We show some of its properties, provide a characterization and explore some of its subclasses. Furthermore, we provide a complete classification of regular simple Hom-Lie algebras over any arbitrary field, together with a description of a lower bound of the number of their isomorphism classes, which depends entirely on the finiteness or not of the underlying field. In addition, we establish that every simple anticommutative algebra of dimension $3$ turns out to be the outside Yau's twist of the special orthogonal Lie algebra $\mathfrak{so}(3,\F)$ with respect to some bijective linear map. Also, we determine all the simple Hom-Lie algebras of dimension $2$, up to conjugacy, which were wrongly claimed to be nonexistent in previous literature. Finally, we establish a new \emph{simplicity criterion} for Lie algebras, which as an application shows that the simplicity in the category of multiplicative Hom-Lie algebras is equivalent to that of anticommutative algebras.
\end{abstract}

%\subjclass[2020]{17B61, 17A30, 17B20, 17B50, 17B60, 20G15}

\address{Department of Mathematics, Faculty of Sciences, University of Moulay Ismail, P.O. 11201 Zitoune, Meknes, Morocco}
\email{yns.elkharraf@gmail.com}
\date{\today}
\maketitle
\sqsp

%%
%%  Introduction
%%

\section{Introduction}

For a long time, classification problems have arisen as natural quests with the wide expansion of mathematics. Such investigation often leads to a better and a deeper comprehension of the subject and frequently gives birth to new ideas and techniques, and even in some cases, allows breakthroughs on some hard questions, e.g \emph{Two Generators Theorem of simple Lie algebras in arbitrary characteristics}, cf. \cite{Boi09}. If we quoted some examples here of achieved classifications, there would be two : One, \emph{the classification of finite simple groups}, which is considered as one of the great intellectual achievements of humanity, and the other is \emph{the classification of simple Lie algebras over fields of characteristic different from $2$ and $3$}, which is a landmark achievement of modern mathematics. There is no need to recall the well-known roles and the importance that they play in almost every area of mathematics, especially in geometry and physics.

Hoping to do the same for the class of simple Hom-Lie algebras, even partially. Unfortunately, we show the existence of many aberrant cases, which makes a classification seems to be out of reach once it is considered without any further restrictions. 

In 2016 in \cite{CH16}, X. Chen and W. Han achieve the classification of regular simple Hom-Lie algebras over the field of complex numbers $\C$. However, their classification could be straightforward extended to any field of characteristic zero, as we will show it in the sequel. Further, by Block’s fundamental result on differentiably simple algebras, we could determine the structure theorem of regular simple Hom-Lie algebras over an arbitrary field. In addition, we describe the number of isomorphism classes of regular simple Hom-Lie structures on a given direct sum of copies of a simple Lie algebra. In fact we prove that this number depends entirely on the (in)finiteness the underlying field, i.e. it is finite if and only if the underlying field is finite too.

Also, we determine all the $2$-dimensional simple Hom-Lie algebras, which were wrongly claimed to be nonexistent in previous literature. On the other hand, we introduce a new interesting class of algebras that we call \emph{strongly simple Hom-Lie algebras}, it is defined as the class of anticommutative algebras that are simple Hom-Lie algebras with respect to every twisting map, i.e. every homomorphism of vector space satisfying the Hom-Jacobi identity. In this work, we characterize this class and we describe some of its properties along with some of its subclasses.  

Furthermore and as an illustration of our motivation in the classification problem of simple Hom-Lie algebras, we were able to establish a new simplicity criterion for Lie algebras, which has an immediate application by allowing us to show that a strongly simple Hom-Lie algebra is equivalent to a multiplicative simple Hom-Lie algebra, also equivalent to a simple anticommutative algebra. So that the definition of simplicity in the category of multiplicative Hom-Lie algebras coincides with the simplicity in the category of anticommutative algebras.

The article is organized as follows. In section $2$, we summarize the basic definitions and introduce the notions of inside and outside Yau's twists with respect to a linear map. In section $3$, we determine all the $2$-dimensional simple Hom-Lie algebras. In section $4$, we revisit the classification of multiplicative simple Hom-Lie algebras and extend it to any arbitrary field. In section $5$, we describe the number of isomorphism classes of regular simple Hom-Lie structures over a pair $(\g,n)$, where $\g$ is a simple Lie algebra and $n$ is the number of its copies, by different approaches. In section $6$, we provide a new simplicity criterion for Lie algebras. We construct many families of nonsimple anticommutative algebras that can be endowed with the structure of simple Hom-Lie algebras, then we introduce the class of strongly simple Hom-Lie algebras and we characterize it. The last section is devoted to some interesting subclasses of strongly simple Hom-Lie algebras. 

Throughout this article, there is no restriction on the field $\F$ and all the algebras are considered to be finite dimensional, unless it is stated otherwise.

\section{Preliminaries }
The notion of Hom-Lie algebras was first introduced by J. T. Hartwig, D. Larsson, S. D. Silvestrov in \cite{HLS06}, as it naturally emerged from the $q$-deformations of Witt and Virasoro algebras when substituting the one parameter $q$ with a homomorphism of vector space $\sigma$. 
Recall that an anticommutative algebra is an algebra such that its multiplication satisfies the anticommutativity condition, i.e. $x^2=0,$ for all $x$.
\begin{defn}\label{d1}
A \emph{Hom-Lie} algebra is a triple $(A,[\cdot,\cdot],\sigma)$ consisting of an anticommutative algebra $(A,[\cdot,\cdot])$ and a linear map $\sigma$ satisfying
  $$J_{\sigma,\brk}=\cycl_{x,y,z}{[\sigma(x),[y,z]]}=0,\quad \forall x,y,z\in A,\hspace{1.4cm} (\text{Hom-Jacobi identity})$$
where $J_{\sigma,\brk}: A^{\wedge^3}\longrightarrow A$ is a trilinear map, called the Hom-Jacobian.
The linear map $\sigma$ is called the \emph{twisting map} of the Hom-Lie algebra $A$. The set of all twisting maps on $A$ is noted by $\HS(A)$.
\end{defn}
\begin{rem}
There is a natural embedding of Lie algebras in the category Hom-Lie algebras, by augmenting them with the identity map.

\end{rem}
 
\begin{defn}
 A homomophism of Hom-Lie algebras $$\varphi:~(A,[\cdot,\cdot]^{A},\sigma_{A})\longrightarrow (B,[\cdot,\cdot]^{B},\sigma_{B})$$ is a homomorphism of the underlying anticommutative algebras (i.e. $\varphi\circ[\cdot,\cdot]^{A}=[\cdot,\cdot]^{B}\circ\varphi^{\otimes^2}$) satisfying, in addition, the compatibility condition : $\varphi\circ\sigma_{A}=\sigma_{B}\circ\varphi $.
\end{defn}

\begin{defn}
   Let $(A,[\cdot,\cdot],\sigma)$ be a Hom-Lie algebra. It is called to be
   \begin{itemize}
      \item \emph{Multiplicative}, if $ \sigma\circ\brk=\brk\circ\sigma^{\otimes^2}$. 
   \item \emph{Regular}, if $\sigma$ is an automorphism of anticommutative algebras.
   \end{itemize}
\end{defn}
  To avoid ambiguity, we introduce the notion of Hom-ideal.
\begin{defn}
  A Hom-ideal $I$ is an ideal in the usual sense which, in addition, is left invariant by the twisting map $\sigma$, i.e
  $[I,A]\subset I~~\text{and}~~\sigma(I)\subset I$.
\end{defn}
  
\begin{defn}
  A Hom-Lie algebra $(A,\left[\cdot,\cdot\right],\sigma)$ is called simple, if it is not abelian and has no proper Hom-ideal. It is called semisimple if its Hom-radical is zero, where a Hom-radical is the maximal solvable Hom-ideal.
\end{defn}

\begin{prop}[Yau's twisting principle]
 Let $(A,\brk,\sigma)$ be a Hom-Lie algebra and $\theta\in\End(A)$ such that $\theta\circ\brk=\brk\circ\theta^{\otimes^2}$. Then, $(A,\brk_\theta,\sigma_\theta)$ is also a Hom-Lie algebra, where $\brk_\theta:=\theta\circ\brk$ and $\sigma_\theta:=\theta\circ\sigma$.\\ In addition, if $\theta$ is invertible. Then, $(A,\brk,\sigma)$ is multiplicative (resp. regular) if and only if $(A,\brk_\theta,\sigma_\theta)$ is multiplicative (resp. regular).
\end{prop}
\begin{proof}
\begin{align*}
    J_{\sigma_\theta,\brk_\theta}=\theta^2\circ J_{\sigma,\brk}=0
\end{align*}
When $\theta$ is invertible, it's clear that the multiplicativity and the regularity are preserved by the Yau's twist of Hom-Lie algebras. Conversely, it suffices to consider the Yau's twist with respect to the inverse $\theta^{-1}$.
\end{proof}
\begin{rem}
We shall said, in general, that an algebra $(A,\brk')$ is the outside (resp. inside) Yau's twist of a Lie algebra $(A,\brk)$ with respect to a linear map $\theta$ if its multiplication is given by $\brk'=\theta\circ\brk$ (resp. $\brk'=\brk\circ\theta^{\otimes^2}$).
When $\theta$ is multiplicative the both notions coincide with the classical Yau's twist described above, see also \cite[Theorem 2.4]{Yau09}.
\end{rem}

\section{Classification of 2-dimensional simple Hom-Lie algebras} 

Here, we determine all $2$-dimensional simple Hom-Lie algebras up to isomorphism, but first, we describe the phenomena that over some nonalgebraically closed fields there exist abelian algebras of dimension greater than one such that, with respect to some twisting maps, their only Hom-ideals are the trivial ones.

\begin{prop}\label{p0}
There are infinitely many fields over which the abelian two-dimensional algebra $\mathfrak{a}_2$ has only trivial Hom-ideals with respect to some twisting map. 
\end{prop}
\begin{proof}
Every proper nontrivial ideal of the $2$-dimensional abelian algebra $\mathfrak{a}_2$ is spanned by a nonzero vector. Thus, a twisting map $\sigma\in\End_\F(\mathfrak{a}_2)$, which makes the only Hom-ideals of $\mathfrak{a}_2$ to be the trivial ones, must have no eigenvectors. Hence, $\F$ is not algebraically closed. It is obvious that $\HS(\mathfrak{a}_2)=\End(\mathfrak{a}_2)$ for dimensional reason. Let $\sigma$ be such that $B=\{e_1,\sigma(e_1)\}$ is a basis for some vector $e_1\in\mathfrak{a}_2$. Hence, the matrix of $\sigma$ with respect to this basis is of the form
$$M_B(\sigma)=
\begin{pmatrix}
0 &a_0\\
1&a_1
\end{pmatrix}
$$
Clearly $a_0\neq0$, because if not $\sigma(e_1)$ would be an eigenvector of $\sigma$. Let's suppose that there exists a nonzero vector $x=\alpha e_1+\beta \sigma(e_1)$ such that $\exists\lambda\in\F,~\sigma(x)=\lambda x$, then 
\begin{align}\label{eq1}
    \left\{
    \begin{array}{ll}
           \lambda\alpha=a_0\beta\\
           \lambda\beta=a_1\beta+\alpha
           \end{array}
    \right.
\end{align}
Thus $\alpha\beta\neq0$ and so do $\lambda\neq0$ which in addition must satisfy the following quadratic equation 
\begin{align}\label{eq2}
    \lambda^2-a_1 \lambda-a_0=0
\end{align}
For finite fields, a simple counting argument suffices. Indeed, there is $|\F|=q$ monic polynomials of degree $1$, $X-c_0,~c_0\in\F$, and similarly there is $q^2$ monic polynomials of degree $2$, $X^2+c_1 X+c_0\in\F[X]$. By the commutativity there is only $\frac{q^2+q}{2}$ distinct products of monic polynomials of degree $1$. Hence, there is $\binom{q}{2}$ irreducible monic quadratic polynomials and therefore there exist $a_1$ and $a_0$ such that $(\ref{eq2})$ has no solution.
For infinite fields, $\mathfrak{a}_2$ has a twisting map with no eigenvectors if and only if the field $\F$ admits at least an irreducible quadratic polynomial. In both cases, such fields are infinitely many and in characteristic zero they are uncountably many. 
\end{proof}

\begin{cor}
Over algebraically closed fields, any abelian algebra of dimension greater than one, has proper nontrivial Hom-ideals with respect to any linear map.
\end{cor}

\begin{rem}
In the definition of simplicity in the category of Lie algebras, or generally in anticommutative algebras, the assumption to be not abelian, has for unique purpose to avoid that the one-dimensional algebra to be considered as simple. But, for the case of simple Hom-Lie algebras, the situation is slightly different as this assumption excludes, effectively, all the abelian algebras, cf. the above Proposition \ref{p0}.
\end{rem}

\begin{thm}\label{t0}
In dimension $2$, the only nonabelian anticommutative algebra, up to isomorphism, is the affine Lie algebra $\mathfrak{aff}(\F)$, which can be endowed with the structure of simple Hom-Lie algebra whenever the twisting map does not leave invariant its unique one-dimensional 
ideal.
\end{thm}
\begin{proof}
It's an easy task to show that every $2$-dimensional nonabelian anticommutative algebra is isomorphic to the affine Lie algebra $\mathfrak{aff}(\F):~[x,y]=y $. Then since $y$ spans the only nontrivial proper ideal, having the structure of a simple Hom-Lie algebra on $\mathfrak{aff}(\F)$ is reduced to take a twisting map $\sigma$ such that $\sigma(y)\notin\F y$, which is possible because the Hom-Jacobian is an alternating $3$-linear map and hence $\HS(\mathfrak{aff}(\F))=\mathrm{End}(\mathfrak{aff}(\F))$.
\end{proof} 

\begin{cor}
There is no multiplicative simple Hom-Lie algebra of dimension $2$.
\end{cor}
\begin{proof}
Let $I$ be the unique $1$-dimensional ideal of $\mathfrak{aff}(\F)$, then $\sigma(I)=\sigma([I,\mathfrak{aff}(\F)])= [\sigma(I),\sigma(\mathfrak{aff}(\F))]\subset I$. Hence $I$ is a proper nontrivial Hom-ideal of $\mathfrak{aff}(\F)$. 
\end{proof}

\begin{thm}
There is at least as much nonisomorphic simple Hom-Lie structures on $\mathfrak{aff}(\F)$ as elements in $\F$.
\end{thm}
\begin{proof}
Set $M_{\sigma}=\begin{pmatrix}\sigma_{11} & \sigma_{12}\\ \sigma_{21} & \sigma_{22}\end{pmatrix}$ and $M_{\eta}=\begin{pmatrix}\eta_{11} & \eta_{12}\\ \eta_{21} & \eta_{22}\end{pmatrix}$ be respectively the matrices of two twisting maps $\sigma$ and $\eta$  of $\mathfrak{aff}(\F):~[x,y]=y$, with respect to the basis $B=\left\{x,y\right\}$. They induce simple Hom-Lie structures if and only if $\sigma_{12} \eta_{12}\neq  0$.\\
Let $\varphi\in\operatorname{Aut}(\mathfrak{aff}(\F))$ and $M_{\varphi}=\begin{pmatrix}a& b\\c & d\end{pmatrix}$ be its matrix of $\varphi$ with respect to $B$. So, 
\begin{align*}
    \varphi\in\operatorname{Aut}(\mathfrak{aff}(\F))&\iff\left[\varphi(x),\varphi(y)\right]=\varphi(y)~\text{and}~\det(\varphi)\neq 0\\
    &\iff  M_{\varphi}=\begin{pmatrix}a& b\\c & d\end{pmatrix},~a d= d,~b=0~\text{and}~ a d\neq 0\\
    &\iff  M_{\varphi}=\begin{pmatrix}1& 0\\c & d\end{pmatrix},~c,d\in\F~\text{and}~ d\neq 0
\end{align*}
Then $\varphi$ is an isomorphism between $(\mathfrak{aff}(\F),\brk_\sigma,\sigma)$ and $(\mathfrak{aff}(\F),\brk_\eta,\eta)$, if in addition $            \varphi\circ\sigma=\eta\circ\varphi $. 
Hence, 

\begin{align*}
            \left\{
            \begin{array}{ll}
           \varphi\in\operatorname{Aut}(\mathfrak{aff}(\F))\\
            \varphi\circ\sigma=\eta\circ\varphi 
                \end{array}
              \right.&\iff  \begin{pmatrix}1& 0\\c & d\end{pmatrix}\begin{pmatrix}\sigma_{11} & \sigma_{12}\\ \sigma_{21} & \sigma_{22}\end{pmatrix}=\begin{pmatrix}\eta_{11} & \eta_{12}\\ \eta_{21} & \eta_{22}\end{pmatrix}\begin{pmatrix}1 & 0\\ c & d\end{pmatrix},~\sigma_{12} \eta_{12} d\neq 0\\[2mm]
              &\iff \left\{
            \begin{array}{ll}
           d=\dfrac{\sigma_{12}}{\eta_{12}},~~\sigma_{12} \eta_{12} \neq 0\\[4mm]
           c=\dfrac{\sigma_{11}-\eta_{11}}{\eta_{12}}\\[4mm]
            \eta_{11}+\eta_{22}=\sigma_{11}+\sigma_{22}\\[2mm]
            \sigma_{12} \sigma_{21}-\eta_{12} \eta_{21}=(\sigma_{11}-\eta_{11})(\sigma_{22}-\eta_{11})
            \end{array}
              \right.
\end{align*}
Hence the two Hom-structures are isomorphic if and only if 
\begin{align*}
\left\{
\begin{array}{ll}
               \operatorname{tr}(\sigma)=\operatorname{tr}(\eta)\\[3mm]
               \sigma_{12} \sigma_{21}-\eta_{12} \eta_{21}=(\sigma_{11}-\eta_{11})(\sigma_{22}-\eta_{11})
     \end{array}
     \right.
\end{align*}
Thus there is at least as much nonisomorphic simple Hom-Lie structures on $\mathfrak{aff}(\F)$ as elements in $\F$, since the trace is invariant by matrix similarity.
\end{proof}

%%%%%%%%%%%%%%%%%%%%%%%%%%%%%%%%%%%%%%%%%%%%%%%%%%%%%%%%%%%%%%%%%%%%%%%%%%%%%%%%%%%%%%%%%%%%%%%

\section{Classification of Multiplicative Simple Hom-Lie algebras}
In \cite{CH16}, \emph{Classification of multiplicative simple Hom-Lie algebras}, the authors only classified regular simple Hom-Lie algebras over the field of complex numbers $\C$. Here, we extend the classification to any arbitrary field $\F$ and also we give a characterization of the multiplicative singular case.

\begin{lem}\label{l41}
Let $(A,[\cdot,\cdot],\sigma)$ be a multiplicative simple Hom-Lie algebra, then 
$\sigma\in\Aut(A)$ or $\sigma=0$.
\end{lem}
\begin{proof}
By multiplicativity $\Ker(\sigma)$ is a Hom-ideal. Hence, by simplicity, it is either zero or equal to $A$.
\end{proof}

\begin{cor}\label{c-1}
A multiplicative simple Hom-Lie algebra is either a regular simple Hom-Lie algebra or a simple anticommutative algebra when the twisting map is the zero map.
\end{cor}

\begin{lem}\label{l2}
Let $(A,[\cdot,\cdot],\sigma)$ be a regular Hom-Lie algebra, 
then $(A,[\cdot,\cdot]_{\sigma^{-1}})$ is its induced Lie algebra and $\sigma$ is an automorphism of Lie algebras.
\end{lem}

\begin{proof}We have 
$\cycl_{x,y,z}[\sigma(x),[y,z]]=\cycl_{x,y,z}\sigma^2\big(\sigma^{-1}\circ[x,\sigma^{-1}\circ[y,z]]\big)=0$ implies $\cycl_{x,y,z}[x,[y,z]_{\sigma^{-1}}]_{\sigma^{-1}}=0$. In addition, $\sigma$ is automatically an automorphism of Lie algebras.
\end{proof}

\begin{rem}
For regular Hom-Lie algebras, $I$ is a Hom-ideal if and only if it is an ideal of its induced Lie algebra that is stable by its twisting map.
\end{rem}

\begin{lem}\label{l3}
$\varphi :  (A,[\cdot,\cdot]^{A},\sigma_{A})\longrightarrow (B,[\cdot,\cdot]^{B},\sigma_{B})$ is a homomorphism of regular Hom-Lie algebras, if and only if $\varphi : (A,[\cdot,\cdot]^{A}_{\sigma_{A}^{-1}})\longrightarrow(B,[\cdot,\cdot]^{B}_{\sigma_{B}^{-1}})$ is a homomorphism of Lie algebras and $\varphi\circ\sigma_{A}=\sigma_{B}\circ\varphi$.
\end{lem}
\begin{proof}
\begin{align*}\vspace{-2cm}
    \varphi:(A,[\cdot,\cdot]^{A},\sigma_{A})\longrightarrow (B,[\cdot,\cdot]^{B},\sigma_{B})& \iff\left\{
    \begin{array}{ll}
           \varphi\circ[\cdot,\cdot]^{A}=[\cdot,\cdot]^{B}\circ\varphi\otimes\varphi\\
           \varphi\circ\sigma_{A}=\sigma_{B}\circ\varphi
    \end{array}
    \right.\\
     &\hspace{-5cm} \iff\left\{
    \begin{array}{ll}
         \varphi\circ\sigma_{A}^{-1}\circ[\cdot,\cdot]^{A}\circ\sigma_{A}\otimes\sigma_{A}=\sigma_{B}^{-1}\circ[\cdot,\cdot]^{B}\circ\sigma_{B}\otimes\sigma_{B}\circ\varphi\otimes\varphi\\
    \varphi\circ\sigma_{A}=\sigma_{B}\circ\varphi
    \end{array}
    \right.\\
    &\hspace{-5cm} \iff\left\{
    \begin{array}{ll}
         \varphi\circ[\cdot,\cdot]^{A}_{\sigma_{A}^{-1}}\circ\sigma_{A}\otimes\sigma_{A}=[\cdot,\cdot]^{B}_{\sigma_{B}^{-1}}\circ\varphi\otimes\varphi\circ\sigma_{A}\otimes\sigma_{A}\\
    \varphi\circ\sigma_{A}=\sigma_{B}\circ\varphi
    \end{array}
    \right.\\
    &\hspace{-5cm} \iff\left\{
    \begin{array}{ll}
          \varphi\circ[\cdot,\cdot]^{A}_{\sigma_{A}^{-1}}=[\cdot,\cdot]^{B}_{\sigma_{B}^{-1}}\circ\varphi\otimes\varphi\\
          \varphi\circ\sigma_{A}=\sigma_{B}\circ\varphi
    \end{array}
    \right.\\
     &\hspace{-5cm}\iff\left\{
    \begin{array}{ll}
              \varphi:(A,[\cdot,\cdot]^{A}_{\sigma_{A}^{-1}})\longrightarrow (B,[\cdot,\cdot]^{B}_{\sigma_{B}^{-1}})\\
       \varphi\circ\sigma_{A}=\sigma_{B}\circ\varphi
    \end{array}
    \right.
\end{align*}
\end{proof}
\begin{thm}\label{t2}
A regular simple Hom-Lie algebra $(A,[\cdot,\cdot],\sigma)$ is a direct sum of $n$-copies of a simple Lie algebra $\s$ such that $A=\bigoplus_{i=1}^n\s_i,$  $(\s_i,[\cdot,\cdot]_i)\simeq(\s,[\cdot,\cdot]_\s)$, $[\cdot,\cdot]=\sigma\circ\oplus_{i=1}^n[\cdot,\cdot]_i$ and the twisting map is of the form $\sigma=\oplus_{i=1}^n\sigma_{i\rho(i)}$ where $\sigma_{i\rho(i)}:\s_i\longrightarrow\s_{\rho(i)}$ are automorphisms of the simple Lie algebra $\s$ and $\rho$ is a cyclic permutation of length $n$.
\end{thm}

\begin{proof}
Following Lemma \ref{l2}, $(A,[\cdot,\cdot]_{\sigma^{-1}})$ is the induced Lie algebra of the regular simple Hom-Lie algebra $(A,[\cdot,\cdot],\sigma)$ and let $\mathrm{Rad}(A)$ be its radical. By multiplicativity, $\sigma(\mathrm{Rad}(A))$ is also a solvable ideal, thus $\sigma(\mathrm{Rad}(A))\subset\mathrm{Rad}(A)$ and, by injectivity of $\sigma$, the equality holds $\sigma(\mathrm{Rad}(A))=\mathrm{Rad}(A)$. Hence, $\mathrm{Rad}(A)$ is an ideal of $(A,\brk)$ and $\sigma$-stable, thus a Hom-ideal. So by Hom-simplicity either $\mathrm{Rad}(A)=0$ or $\mathrm{Rad}(A)=A$. However, in the multiplicative case, the derived series of a Hom-ideal are Hom-ideals. Therefore $\mathrm{Rad}(A)$ must be zero. Hence $(A,[\cdot,\cdot]_{\sigma^{-1}})$ is a semisimple Lie algebra and $\sigma$ is an automorphism of Lie algebras. Now, let $\m$ be a nontrivial minimal ideal of $(A,[\cdot,\cdot]_{\sigma^{-1}})$, then the finite dimensionality implies that there exists a smallest integer $n\geq1$ such that $\sigma^n(\m)\cap\bigoplus_{i=0}^{n-1}\sigma^{i}(\m)$ is a nonzero ideal, which is clearly contained in the minimal ideal $\sigma^n(\m)$, thus $\sigma^n(\m)\subset\bigoplus_{i=0}^{n-1}\sigma^{i}(\m)$. So $\bigoplus_{i=0}^{n-1}\sigma^i(\m)$ is a $\sigma$-stable ideal and thus is equal to the whole $A$. Further, for any $x\in\m$ there exist $x_i\in\m,~0\leq i\leq n-1$, such that $\sigma^{n}(x)=\sum_{i=0}^{n-1}\sigma^i(x_i)$, so that  $[\sigma^{n}(\m),\sigma^{n}(x)]_{\sigma^{-1}}=[\sigma^{n}(\m),\sum_{i=0}^{n-1}\sigma^i(x_i)]_{\sigma^{-1}}=\sum_{i=0}^{n-1}[\sigma^{n}(\m),\sigma^i(x_i)]_{\sigma^{-1}}\subset\bigoplus_{i=0}^{n-1}\sigma^{n}(\m)\cap\sigma^i(\m)$. Therefore by the minimality of $\m$ and the smallness of $n$, $\sigma^n(\m)=\m$.\\
If $\ch=p>0$ then thanks to the structure theorem of semisimple Lie algebras in prime characteristic due to R. E. Block in \cite[p. 457]{Blo69}, a semisimple Lie algebra has its minimal ideals in direct sum, called the socle of the semisimple Lie algebra, and a minimal ideal is of the form $\s\otimes\mathcal{O}(m ; \underline{1})$, where $m\geq0$, $\underline{1}=(1,\cdots,1)$ a $m$-tuple, $\s$ is simple Lie algebra and $\mathcal{O}(m ; \underline{1})$ is isomorphic to the truncated polynomial algebra $\F[X_1,...,X_m]/(X^p_1,\cdots,X^p_m)$. Thus the socle of the induced Lie algebra $(A,\brk_{\sigma^{-1}})$ coincide with $A$ and is equal to $n$-copies of a minimal ideal of the form $\m=\s\otimes\mathcal{O}(m ; \underline{1})$. However, this situation forces $m$ to be zero, since $\s\otimes\mathcal{O}(m ; \underline{1})$ has a maximal nontrivial nilpotent ideal if $m\geq1$. So the minimal ideal $\m$ is a simple Lie algebra.
\\
On the other hand, if $\ch=0$ then $A$ is a direct sum of $n$-copies of a simple Lie algebra, thanks to structure theorem in \cite[p. 71]{Jac79}.\\
Finally, we can summarize the both cases as follows :  $A=\bigoplus_{i=1}^n\s_i$ where $[\cdot,\cdot]=\sigma\circ\oplus_{i=1}^n[\cdot,\cdot]_i$, $\sigma^n(\s_i)=\s_i,~1\leq i\leq n$,  $(\s_i,[\cdot,\cdot]_i)\simeq(\s,[\cdot,\cdot]_\s)$ is %a minimal ideal of a semisimple Lie algebra,
a simple Lie algebra, and so $\sigma=\oplus_{i=1}^n\sigma_{i\rho(i)}$ with $\sigma_{i\rho(i)}:=\sigma_{\mid \s_i} : \s_i\longrightarrow \s_{\rho(i)}$ are automorphisms of the Lie algebra $\s$ where $\rho$ is clearly a cyclic permutation of length $n$.
\end{proof}

\begin{rem}\label{r0}
The dimension of a regular simple Hom-Lie algebra is a multiple of the dimension of a simple Lie algebra. 
\end{rem}

\begin{thm}\label{t3}
  A multiplicative simple Hom-Lie algebra $(A,[\cdot,\cdot],\sigma)$ is either is the Yau's twist of a direct sum of $n$-copies of 
  a simple Lie algebra $\s$, with respect to an automorphism $\sigma$ acting simply transitive on the $n$-components, up to the conjugacy of $\sigma_{\mid\s}^n$, or a simple anticommutative algebra up to isomorphism of anticommutative algebras, if the twisting map is zero.
\end{thm}

\begin{proof}
For $\sigma=0$, the result is immediate. For the regular case, by the structure theorem of regular simple Hom-Lie algebras, cf. the above Theorem \ref{t2}, tells that $A=\bigoplus_{i=1}^n\s_i$ is a direct sum of a simple Lie algebra $\s$ up to isomorphism, where $\s_i\simeq\s$, $\sigma_{\mid \s_i}=\sigma_{i\rho(i)}: \s_i\longrightarrow \s_{\rho(i)}$ and $\sigma=\oplus_{i=1}^n\sigma_{i\rho(i)}$, and let $B=\bigoplus_{i=1}^n\widetilde{\s}_i$ where $\widetilde{\s}_i\simeq\widetilde{\s}$, $\eta_{\mid \widetilde{\s}_i}=\eta_{i\widetilde{\rho}(i)}: \widetilde{\s}_i\longrightarrow \widetilde{\s}_{\widetilde{\rho}(i)}$ and $\eta=\oplus_{i=1}^n\eta_{i\widetilde{\rho}(i)}$. So thanks to the Lemma \ref{l3},
\begin{align*}
    \setlength{\arraycolsep}{.9pt}
    (A,\brk,\sigma)\overset{\varphi}{\overset{\simeq}{\longrightarrow}}(B,\widetilde{\brk},\eta)&\iff\left\{
    \begin{array}{ll}
        \varphi \text{ is an isomorphism of Lie algebras}, \\
        \varphi\circ\sigma=\eta\circ\varphi
    \end{array}
\right.\\
&\hspace{-4.2cm}\iff\left\{
    \begin{array}{ll}
        \varphi=\oplus_{i=1}^n\varphi_i,~\varphi_{\mid\s_i}=\varphi_i:\s_i\longrightarrow\widetilde{\s}_{\tau(i)}\text{ is an isomorphism of}\\[1mm]
        \text{ Lie algebras, }\tau\text{ is a cyclic permutation of length }n,\\
        \varphi_{\rho(i)}\circ\sigma_{i\rho(i)}=\eta_{\tau(i)\widetilde{\rho}\circ\tau(i)}\circ\varphi_i,~1\leq i\leq n
    \end{array}
\right.\\
&\hspace{-4.2cm}\iff\left\{
    \begin{array}{ll}
        \varphi=\oplus_{i=1}^n\varphi_i,~\varphi_i:\s_i\longrightarrow\widetilde{\s}_{\tau(i)}\text{ is an isomorphism of Lie algebras},\\[1mm]
        \tau\text{ is a cyclic permutation of length }n,~1\leq i\leq n\\
        \varphi_{i}\circ\sigma_{\rho^{n-1}(i)i}\circ\sigma_{i\rho(i)}\cdots\sigma_{\rho^{n-2}(i)\rho^{n-1}(i)}=\\
        \hspace{2cm}\eta_{\widetilde{\rho}^{n-1}(\tau(i))\tau(i)}\circ\eta_{\tau(i)\widetilde{\rho}(\tau(i))}\cdots\eta_{\widetilde{\rho}^{n-2}(\tau(i))\widetilde{\rho}^{n-1}(\tau(i))}\circ\varphi_i
    \end{array}
\right.\\
&\hspace{-4.2cm}\iff\left\{
    \begin{array}{ll}
        \varphi=\oplus_{i=1}^n\varphi_i,~\varphi_i:\s_i\longrightarrow\widetilde{\s}_{\tau(i)}\text{ is a isomorphism of Lie algebras},\\[1mm]
        \tau\text{ is a cyclic permutation of length }n,~\tau\circ\rho=\widetilde{\rho}\circ\tau,\\
        \varphi_{i}\circ\sigma^n_{\mid\s_i}=\eta_{\mid\varphi_i(\s_i)}^n\circ\varphi_i,~1\leq i\leq n
    \end{array}
\right.\\
&\hspace{-4.2cm}\iff\left\{
    \begin{array}{ll}
       \exists i,~ \varphi_i:\s_i\longrightarrow\widetilde{\s}_{\tau(i)}\text{ an isomorphism of Lie algebras},\\[1mm]
       \tau\text{ is a cyclic permutation of length }n,\\
        \varphi_{i}\circ\sigma^n_{\mid\s_i}=\eta_{\mid\varphi_i(\s_i)}^n\circ\varphi_i
    \end{array}
\right.
\end{align*}

To convince ourselves that the last equivalence holds. It suffices to note that one can recover $\varphi$ from the following commutative diagram, once $\varphi_i$, for any $i$, is given 
$$\xymatrix@R+2pc@C+4pc{\s_i\ar@{->}[r]^{\sigma_{i\rho(i)}}\ar[d]_{\varphi_{i}}&\s_{\rho(i)}\ar[d]^{\varphi_{\rho(i)}}\\
\widetilde{\s}_{\tau(i)}\ar[r]^{\eta_{\tau(i)\widetilde{\rho}\circ\tau(i)}}&\widetilde{\s}_{\widetilde{\rho}\circ\tau(i)}}
$$

Conclusion, $(A,[\cdot,\cdot],\sigma)$ and $(B,\widetilde{\brk},\eta)$ are isomorphic if and only if there exists an isomorphism $\phi$ of the simple Lie ideals $\s$ and $\widetilde{\s}$ of their induced respective Lie algebras such that $\sigma^n_{\mid\s}=\phi^{-1}\circ\eta^n_{\mid\widetilde{\s}}\circ\phi$, equivalently, if there exist $i$ and $j$ such that $\sigma^n_{\mid\s_i}$ and $\eta^n_{\mid\widetilde{\s}_j}$ are conjugate. Therefore a regular simple Hom-Lie structure is unique up to the conjugacy of the $n$th power of its twisting map on any of its simple Lie ideals.
\end{proof}
\begin{rem}
We shall identify the regular simple Hom-Lie algebra $(\bigoplus_{i=1}^n\s_i,\brk_\sigma,\sigma)$ with the triple $(\s,n,[\sigma^n_{\vert\s}])$, where $\s\simeq\s_i$ and $[\sigma^n_{\vert\s}]$ is the conjugacy class of $\sigma^n_{\vert\s}$ in $\Aut(\s)$.
\end{rem}
%%%%%%%%%%%%%%%%%%%%%%%%%%%%%%%%%%%%%%%%%%%%%%%%%%%%%%%%%%%%%%%

\section{Isomorphism classes of regular simple Hom-Lie algebras}

Let $\g$ be a complex simple Lie algebra of rank $l$, $\mathfrak{h}$ a Cartan subalgebra and $\Delta$ a root system with $\p$ its base of simple roots. We have the Weyl group $W$ generated by simple reflections $\operatorname{r}_{\alpha^\vee}=\operatorname{id}-\langle\cdot,\alpha^\vee\rangle\,\alpha^\vee$, where $\alpha^\vee$ is the dual root associated to $\alpha$.

Let $-\delta$ (often labelled $\alpha_0$) be the lowest root of $\Delta$ with respect to the partial order defined by $\p$ (i.e. $\alpha\geq\beta$ if and only if $\alpha-\beta$ is the sum of simple roots or $\alpha=\beta$).
The system of vectors $\tp=\{\alpha_0,\alpha_1,\cdots,\alpha_l\}$ is called admissible, since $\langle \gamma,\mu\rangle$ is a nonpositive integer for any two distinct vectors $\mu,\gamma\in\tp$. It is dubbed the extended system of simple roots of $\Delta$ and is described and classified by the extended Dynkin diagrams, cf. \cite[Table 3, p.~228]{GOV94} or \cite[Table 8, p.~145]{Dyn57}.

 The domain relative to $\tp$
 \begin{align*}
    D_{\tp}:=\left\{         x\in\mathfrak{h}~\Big\vert\begin{array}{lll}
&\mathrm{Re}\big(\alpha(x)\big)\geq0,\text{ and whenever }\mathrm{Re}\big(\alpha(x)\big)=0&\\
    &\text{then }\im\big(\alpha(x)\big)\geq0\text{ for all } \alpha\in\{1+\alpha_0\}\cup\p  & 
    \end{array}
\right\}
\end{align*}
is a fundamental polyhedron for the action of the 
group the extended Weyl group $Q^\vee\rtimes W$ on $\mathfrak{h}$, called  (where $Q^\vee$ is the dual root
lattice, i.e. a $\mathbb{Z}$-span of the dual root system $\Delta^\vee$), it may also be regarded as the complexification of the (closed) fundamental restricted Weyl chamber relative to $\tp$.

A natural action of the group $\Aut(\tp)$ on $\mathfrak{h}$ is given by $x'=\sigma\cdot x$ if and only if $\alpha_i(x')=\sigma\circ\alpha_i(x),$ for all $\alpha_i\in\tp$.

In \cite[\S 3.6]{GOV94}, V. V. Gorbatsevich, A. L. Onishchik and È. B. Vinberg establish a classification of inner semisimple automorphisms up to conjugacy in $\Aut(\g)$. 

\begin{thm}
Any inner semisimple automorphism of a complex simple Lie algebra $\g$ is conjugate in $\Aut(\g)$ to a unique automorphism of 
the canonical form $\exp 2\pi i x,$ where $x$ is any representative element belonging to $D_{\tp}$ with respect to the action of $\Aut(\tp)$.% on $\mathfrak{h}$.
\end{thm}

Equivalently, inner semisimple automorphisms up to conjugation in $\Aut(\g)$ are in bijective correspondence with the orbits of $D_{\tp}/\Aut(\tp)$.

\begin{cor}
There are uncountably many pairwise nonisomorphic regular simple Hom-Lie structures over a given direct sum of copies of a complex simple Lie algebra.
\end{cor}

Next we follow two other directions. One allows, by direct means, to establish the infinity of the number of regular simple Hom-Lie structures on a direct sum of copies of classical simple Lie algebra over an algebraically closed field of characteristic $\neq 2$ and $3$. The second establishes, by means of results from algebraic geometry, that the finiteness of conjugacy classes in the automorphism group of a nonsolvable Lie algebra over an arbitrary field depends entirely on the finiteness of the underlying field.

\begin{thm}\label{t53}
If $\g$ is a simple Lie algebra over an algebraically closed field of characteristic zero or a classical simple Lie algebra over an algebraically closed field of characteristic $p>3$. Then the trace on the automorphism group $\Aut(\g)$ is surjective.
\end{thm}
\begin{proof}
Let $\mathfrak{h}$ be a Cartan subalgebra of dimension $l$ and $\Delta$ its root system, we have the decomposition 
$$\g=\mathfrak{h}\oplus\bigoplus_{\alpha\in\Delta}\g_\alpha$$
where $\g_\alpha=\{x\in\g\vert~ [h,x]=\alpha(h)x\text{ for all }h\in\mathfrak{h}\}$ is of dimension $1$. For a simple root $\alpha_i\in\p$, we set $\Lambda_{i}=\{\beta=\sum_{j=1}^lm_j(\beta){\alpha_j}\in\Delta\vert~m_i(\beta)\neq0\}$ and $n_i$ its size. Let $c_i$ be a nonzero scalar, with respect to the Chevalley basis $\{x_\alpha,h_i\vert~\alpha\in\Delta\text{ and  }1\leq i\leq\dim\mathfrak{h}\}$, we define the automorphism $\varphi_i$ as follows,
\begin{align*}
    \left\{\begin{array}{ll}
        \varphi_i(h)=h,&\text{for all } h\in\mathfrak{h},\\
         \varphi_i(x_\beta)=c_i^{m_i(\beta)} x_\beta,&\text{for all } \beta=\sum_{j=1}^lm_j(\beta)x_{\alpha_j}.
    \end{array}
    \right.
\end{align*}
Hence,
\begin{align*}
    \operatorname{tr}(\varphi_i)&=\dim\mathfrak{h}+\sum_{\beta\in\Delta}c_i^{m_i(\beta)}\\
    &=\dim\mathfrak{h}+\vert\Delta\vert-n_i+\sum_{\beta\in\Lambda_i}c_i^{m_i(\beta)}\\
    &=\dim\g-n_i+\sum_{\beta\in\Lambda_i}c_i^{m_i(\beta)}
\end{align*}
Let $\p=\{\alpha_1,\cdots,\alpha_l\}$ be a basis of $\Delta$, with the ordering follows that of Dynkin diagrams in \cite{Bou75}. From Hasse diagrams of poset of positive root systems and by induction, one obtains
\begin{itemize}
    \item $\mathrm{A}_l~(l\geq1)$ : for all $1\leq i\leq l$, $$\operatorname{tr}(\varphi_i)=\dim\g+i\left(l+1-i\right)\left(c_i+c_i^{-1}-2\right).$$
    \item $\mathrm{B}_l~(l\geq2)$ : for all $1\leq i\leq l$,
\begin{align*}
       \hspace{-3mm} \operatorname{tr}(\varphi_i)=\dim\g-i\left(4l+1-3i\right)+i\left(2(l-i)+1\right)\left(c_i+c_i^{-1}\right)+\dfrac{(i-1)i}{2}\left(c_i^2+c_i^{-2}\right).
    \end{align*}
    \item $\mathrm{C}_l~(l\geq3)$ : for all $1\leq i\leq l-1$, 
        \begin{align*}
        \left\{
        \begin{array}{ll}
        \operatorname{tr}(\varphi_i)=\dim\g-i\left(4l+1-3i\right)+2i\left(l-i\right)\left(c_i+c_i^{-1}\right)+\dfrac{i(i+1)}{2}\left(c_i^2+c_i^{-2}\right),\\
        \operatorname{tr}(\varphi_l)=\dim\g+\dfrac{l(l+1)}{2}\left(c_l+c_l^{-1}-2\right).
        \end{array}
        \right.
    \end{align*}
    \iffalse
    \begin{align*}
        \left\{
        \begin{array}{ll}
\operatorname{tr}(\varphi_i)=\dim\g-\left(2l+3-i\right)i+\left(c_i+c_i^{-1}\right)\left(l+1-i\right)i+\left(c_i^2+c_i^{-2}\right)\binom{i+1}{2},\\
             \operatorname{tr}(\varphi_l)=\dim\g-\left(c_l+c_l^{-1}-2\right)\left(l+3\right)l.
        \end{array}
        \right.
    \end{align*}
    \fi
    \item $\mathrm{D}_l~(l\geq3)$ : for all $1\leq i\leq l-2$, 
    \begin{align*}
        \left\{
        \begin{array}{ll}
    \operatorname{tr}(\varphi_i)=\dim\g-i\left(4l-1-3i\right)+2i(l-i)\left(c_i+c_i^{-1}\right)+\dfrac{(i-1)i}{2}\left(c_i^2+c_i^{-2}\right),\\
    \operatorname{tr}(\varphi_j)=\dim\g+\dfrac{(l-1)l}{2}\left(c_j+c_j^{-1}-2\right),~\text{ for  }j=l-1\text{ or }l.
        \end{array}
        \right.
    \end{align*}
\item $\mathrm{E}_6$ :  
    \begin{align*}
        \left\{
        \begin{array}{ll}
    \operatorname{tr}(\varphi_i)=46+16\left(c_i+c_i^{-1}\right),\text{ for }i=1\text{ or }6,\\
    \operatorname{tr}(\varphi_i)=28+20\left(c_i+c_i^{-1}\right)+5\left(c_i^2+c_i^{-2}\right),\text{ for }i=3\text{ or }5,\\
    \operatorname{tr}(\varphi_2)=36+20\left(c_2+c_2^{-1}\right)+\left(c_2^2+c_2^{-2}\right),\\
    \operatorname{tr}(\varphi_4)=20+18\left(c_4+c_4^{-1}\right)+9\left(c_4^2+c_4^{-2}\right)+2\left(c_4^3+c_4^{-3}\right).
        \end{array}
        \right.
    \end{align*}
\item $\mathrm{E}_7$ :  
    \begin{align*}
        \left\{
        \begin{array}{ll}
    \operatorname{tr}(\varphi_1)=67+32\left(c_1+c_1^{-1}\right)+\left(c_1^2+c_1^{-2}\right),\\
    \operatorname{tr}(\varphi_2)=49+35\left(c_2+c_2^{-1}\right)+7\left(c_2^2+c_2^{-2}\right),\\
    \operatorname{tr}(\varphi_3)=39+30\left(c_3+c_3^{-1}\right)+15\left(c_3^2+c_3^{-2}\right)+2\left(c_3^3+c_3^{-3}\right)\\
    \operatorname{tr}(\varphi_4)=27+24\left(c_4+c_4^{-1}\right)+18\left(c_4^2+c_4^{-2}\right)+8\left(c_4^3+c_4^{-3}\right)+3\left(c_4^4+c_4^{-4}\right),\\
    \operatorname{tr}(\varphi_5)=33+30\left(c_5+c_5^{-1}\right)+15\left(c_5^2+c_5^{-2}\right)+5\left(c_5^3+c_5^{-3}\right),\\
    \operatorname{tr}(\varphi_6)=49+32\left(c_6+c_6^{-1}\right)+10\left(c_6^2+c_6^{-2}\right),\\
    \operatorname{tr}(\varphi_7)=79+27\left(c_7+c_7^{-1}\right).
    \end{array}
        \right.
    \end{align*}
\item $\mathrm{E}_8$ :  
    \begin{align*}
        \left\{
        \begin{array}{ll}
    \operatorname{tr}(\varphi_1)=&92+64\left(c_1+c_1^{-1}\right)+14\left(c_1^2+c_1^{-2}\right),\\
    \operatorname{tr}(\varphi_2)=&64+56\left(c_2+c_2^{-1}\right)+28\left(c_2^2+c_2^{-2}\right)+8\left(c_2^3+c_2^{-3}\right),\\
    \operatorname{tr}(\varphi_3)=&52+42\left(c_3+c_3^{-1}\right)+35\left(c_3^2+c_3^{-2}\right)+14\left(c_3^3+c_3^{-3}\right)+7\left(c_3^4+c_3^{-4}\right),\\
    \operatorname{tr}(\varphi_4)=&36+30\left(c_4+c_4^{-1}\right)+30\left(c_4^2+c_4^{-2}\right)+20\left(c_4^3+c_4^{-3}\right)\\
    &+15\left(c_4^4+c_4^{-4}\right)+6\left(c_4^5+c_4^{-5}\right)+5\left(c_4^6+c_4^{-6}\right),\\
    \operatorname{tr}(\varphi_5)=&40+40\left(c_5+c_5^{-1}\right)+30\left(c_5^2+c_5^{-2}\right)+20\left(c_5^4+c_5^{-3}\right)\\
    &+10\left(c_5^4+c_5^{-4}\right)+4\left(c_5^5+c_5^{-5}\right),\\
    \operatorname{tr}(\varphi_6)=&54+48\left(c_6+c_6^{-1}\right)+30\left(c_6^2+c_6^{-2}\right)+16\left(c_6^3+c_6^{-3}\right)+3\left(c_6^4+c_6^{-4}\right),\\
    \operatorname{tr}(\varphi_7)=&82+54\left(c_7+c_7^{-1}\right)+27\left(c_7^2+c_7^{-2}\right)+2\left(c_7^3+c_7^{-3}\right),\\
    \operatorname{tr}(\varphi_8)=&134+56\left(c_8+c_8^{-1}\right)+\left(c_8^2+c_8^{-2}\right).
        \end{array}
        \right.
    \end{align*}
\item $\mathrm{F}_4$ :  
    \begin{align*}
        \left\{
        \begin{array}{ll}
    \operatorname{tr}(\varphi_1)=22+14\left(c_1+c_1^{-1}\right)+\left(c_1^2+c_1^{-2}\right),\\
    \operatorname{tr}(\varphi_2)=12+12\left(c_2+c_2^{-1}\right)+6\left(c_2^2+c_2^{-2}\right)+2\left(c_2^3+c_2^{-3}\right),\\
    \operatorname{tr}(\varphi_3)=12+6\left(c_3+c_3^{-1}\right)+9\left(c_3^2+c_3^{-2}\right)+2\left(c_3^3+c_3^{-3}\right)+3\left(c_3^4+c_3^{-4}\right),\\
    \operatorname{tr}(\varphi_4)=22+8\left(c_4+c_4^{-1}\right)+7\left(c_4^2+c_4^{-2}\right).
        \end{array}
        \right.
    \end{align*}
\item $\mathrm{G}_2$ :  
    \begin{align*}
        \left\{
        \begin{array}{ll}
    \operatorname{tr}(\varphi_1)=4+2\left(c_1+c_1^{-1}\right)+\left(c_1^2+c_1^{-2}\right)+2\left(c_1^3+c_1^{-3}\right),\\
    \operatorname{tr}(\varphi_2)=4+4\left(c_2+c_2^{-1}\right)+\left(c_2^2+c_2^{-2}\right).
        \end{array}
        \right.
    \end{align*}
\end{itemize}
\end{proof}
\begin{rem}\label{r54}
If the corresponding group of $\g$ is compact then the trace is bounded and hence not surjective. Particularly, if $\g$ is $\F$-anisotropic over the reals or a $p$-adic field.
%Let $\F$ be a real field or a $p$-adic one. If $\g$ is $\F$-anisotropic (equivalently its corresponding group is compact) then the trace is bounded. 
\end{rem}
\begin{thm}\label{t55}
Let $\g$ be a nonsolvable Lie algebra over a field $\F$. The automorphism group $\Aut(\g)$ has only a finite number of conjugacy classes if and only if $\F$ is finite. Moreover, if $\F$ is perfect then the assumption of nonsolvability can be weakened to non-nilpotency.
\end{thm}
\begin{proof}
Let $G$ be a connected reductive subgroup of $\Aut(\g)$. Over an infinite field $\F$, thanks to \cite[Exp. XIV]{DG70} or \cite[Corollary 18.3]{Bor91}, $G(\F)$ is Zariski-dense in the underlying algebraic group $G$. Therefore, we can assume that $\F$ is algebraically closed without any lost of generality. If the number of characteristic polynomials is finite then by connectedness it is reduced to one, hence every element is unipotent and so $G$ is unipotent, a contradiction. 
Finally, since characteristic polynomial is invariant under similarity, there are infinitely many conjugacy classes in $\Aut(\g)$.
On the other hand, if $\F$ is perfect then by M. Rosenlicht's result on Zariski density in \cite{Ros57}, the assumption of nonsolvability of $\g$ can be weakened to non-nilpotency.
\end{proof}

\begin{cor}
Over an infinite field, there are infinitely many isomorphism classes of regular simple Hom-Lie structures on a given pair $(\g,n)$, i.e a direct sum of $n$-copies of $\g$. On the other hand, over finite fields there are at least two.
\end{cor}

%%%%%%%%%%%%%%%%%%%%%%%%%%%%%%%%%%%%%%%%%%%%%%%%%%%%%%%%%%%%%%%%%%%%%%%%%%%%%%%%%

\section{Strongly Simple Hom-Lie algebras}
In this section, we establish a new simplicity criterion for Lie algebras. Then, we show that a semisimple Hom-Lie algebra needs not to be a direct sum of simple Hom-Lie algebras. Further, we provide many examples of nonsimple, solvable and nilpotent anticommutative algebras that can be endowed with the structure of simple Hom-Lie algebras. Also, we introduce the class of strongly simple Hom-Lie algebras and we characterize it.

\begin{lem}\label{l5}
A simple Lie algebra $\mathfrak{g}$ has the property that for any nonzero vector $x$, $\mathrm{rk} (\ad x)\geq 2$.
%$\dim[x,\mathfrak{g}]\geq 2$.
\end{lem}
\begin{proof}
Let $x$ be a nonzero vector such that %$\dim[x,\mathfrak{g}]\leq1$,
$\mathrm{rk}(\ad x)\leq1$, and as $\mathfrak{g}$ is centerless, then $\mathfrak{g}=\ker \ad x\oplus\F y$ for some $y\neq0$. The adjoint representation endows $\mathfrak{g}$ with the structure of a simple module for the universal enveloping algebra $\mathit{U}(\mathfrak{g})$. Letting $I_\mathfrak{g}$ be the augmentation ideal of $\mathit{U}(\mathfrak{g})$, we have $[y, x]\in I_\mathfrak{g}\cdot x\setminus\{0\}$, so that $I_\mathfrak{g}\cdot x =\mathfrak{g}$. On the other hand, the PBW-Theorem implies $I_\mathfrak{g}\cdot x=\sum_{i\geq1}y^i\mathit{U}(\ker \ad x)\cdot x=\sum_{i\geq1}\F (\ad y)^i(x)\subset(\ad y)(\mathfrak{g})$, a contradiction.
\end{proof}

We can go further and provide a new simplicity criterion for the category of Lie algebras, of arbitrary dimension, which generalizes the above property in Lemma \ref{l5}.

\begin{thm}[Simplicity Criterion]\label{t5} Let $\g$ be a Lie algebra of arbitrary dimension over an arbitrary field. Then, $\g$ is simple if and only if for any nontrivial subspace $S$, either $[S,\g]=\g$ or $[S,\g]\varsubsetneq[[S,\g],\g]$.
\end{thm}
\begin{proof}
Let $S$ be a subspace of $\g$, since $\g$ is perfect and by the Jacobi identity one has $[S,\g]\subset[[S,\g],\g]$. Hence the equality holds either if $S=0$ or $[S,\g]=\g$. Conversely, for a nontrivial ideal $I$, if $[I,\g]=\g$ then $I=\g$, otherwise $[I,\g]\varsubsetneq[[I,\g],\g]\subset [I,\g]$ which is absurd.
\end{proof}

The following  proposition was inspired by the almost classical Lie algebra $\mathfrak{pgl}(n)$.

\begin{prop}\label{p3}
Let $\s$ be a simple Lie algebra over any arbitrary field $\F$. The extended anticommutative algebra $A=\F d\ltimes \s$, where $[d,x]=x$ with $x$ is a nonzero vector of $\s$, is semisimple and has a unique proper nontrivial ideal $\s$ which is a simple ideal of codimension $1$.
\end{prop}
\begin{proof}
$A=\F d\ltimes \s$ with $[d,x]=x$, and let $I$ be a proper nontrivial ideal other than $\s$, so $[I,A]\subset I\cap \s$. So if the intersection is not trivial then, by Theorem \ref{t5}, $\s\subset I$ absurd. Hence, $I=\Span\{d+y\}$ for some nonzero $y\in\s$. Thanks to Lemma \ref{l5}, $\dim[d+y,\s]\geq1$ which again absurd. Therefore $\s$ is the unique proper nontrivial ideal of $A$.
\end{proof}

\begin{rem}\label{r1}
    In general, a semisimple anticommutative algebra needs not to be a direct sum of simple anticommutative algebras over any arbitrary field. The same holds for semisimple Hom-Lie algebras.
\end{rem}

The following proposition provides a family of nonsimple anticommutative algebras which can have the structure of a simple Hom-Lie algebra.

\begin{prop}
  Let $S_n$ be the anticommutative algebra defined by the following relations, with respect to the basis $\{e_1,\cdots,e_n\},~n\ge3$
  $$\left\{\begin{array}{ll}
           ~[e_{i},e_{i+1}]=e_{i+2},~~1\leq i\leq n-2,\\
          ~[e_{n-1},e_{n}]=e_{1}~\text{and }[e_n,e_{1}]=e_{2}
    \end{array}
    \right.
    $$ 
 Then the extension $A_{n+1}=\F d\ltimes S_n$, given by $[d,e_{1}]=e_{2}$,  is a simple Hom-Lie algebra with respect to the twisting map $\sigma$ defined by $\sigma(d)=\sigma(e_i)=0,~1\leq i\leq n-1$ and $\sigma(e_n)=d.$
\end{prop}

\begin{proof}
  Clearly $S_n$ is simple, and thanks to the proof of Proposition \ref{p3}, if for any $x\in S_n\setminus\{0\}$, $\dim[x,S_n]\geq2$ then the extension $A_{n+1}$ has $S_n$ as a unique proper nontrivial ideal, which is the case here. It follows from the construction of $\sigma$ that it does not leave stable $S_n$. Therefore it suffices to check that $\sigma$ is a twisting map. We observe that for any $x,y\in\{e_1,\cdots,e_{n-1}\}\cup\{d\}$, $[x,y]\in \mathcal{C}_{A_{n+1}}(d):=\{z\in A_{n+1}~|~ [z,d]=0\}$, so that
  $\cycl_{e_n,x,y}[\sigma(e_n),[x,y]]=[d,[x,y]]=0$. Hence, $\sigma\in\HS(A_{n+1})$.
\end{proof}

\begin{exs}
\begin{enumerate}
    \item The Heisenberg Lie algebra $\mathfrak{h}_{2n+1}$ generated by the elements $e_i,~1\leq i\leq 2n+1$, satisfying the relations :
    $$[e_{2i-1},e_{2i}]=e_{2n+1},\quad 1\leq i\leq n$$
    is a $2$-step nilpotent Lie algebra that can be endowed with the structure of a simple Hom-Lie algebra. For that, it suffices to take $$\sigma^i(e_{2n+1})=e_{2n+1-i},~1\leq i\leq 2n$$ 
    Indeed, any linear map $\sigma$ hence defined is a twisting map, since every linear map of a metabelian Lie algebra is a twisting map. %hence $\HS(\mathfrak{h}_{2n+1})=\End(\mathfrak{h}_{2n+1})$.
    Then as a consequence of Engel's Theorem (cf. \cite{Sel67}), every ideal intersects nontrivially the center and so the vector $e_{2n+1}$ is included in every nontrivial ideal. Therefore, by the construction of $\sigma$, a nontrivial Hom-ideal must be equal to the whole Lie algebra. 
    \item Let $R_n$ be the anticommutative algebra of dimension $2n\geq4$, defined by : 
$$\left\{\begin{array}{ll}
           ~[e_{2i-1},e_{2i}]=e_{2i+1},~~1\leq i\leq n-1\\
          ~[e_{2n-1},e_{2n}]=e_{1}
    \end{array}
    \right.$$ 
Clearly $R_n$ is solvable and a straightforward verification shows that the linear map $\sigma$ given by $\sigma(e_{2i-1})=e_{2i+2},~1\leq i\leq n-1$, $\sigma(e_{2n-1})=e_{2}$ and $\sigma^2=0$, is a twisting map. Finally, for a nontrivial Hom-ideal $I$, the brackets' form implies that $0\neq[I,R_n]\subset I\cap[R_n,R_n]$, and so at least there exists a vector $e_j\in I,\,j\equiv 1~[2]$. Then by the $\sigma$-invariance of $I$ we must have $I=R_n$. Hence, $R_n$ is a simple Hom-Lie algebra.
    \item The direct sum of the affine Lie algebra $\mathfrak{aff}(\F):[x,y]=y,$ with any abelian algebra $\mathfrak{a}_n$, can be endowed with the structure of simple Hom-Lie algebra. Indeed, it suffices with respect to a basis $\{e_1,\cdots,e_n\}$ of $\mathfrak{a}_n$, to define the desired twisting map $\sigma$ by $\sigma(y)=x,~\sigma(x)=e_1,~\sigma(e_n)=y$ and $\sigma(e_i)=e_{i+1},\,1\leq i\leq n-1$. Any nontrivial Hom-ideal, by construction, must contains $y$ and so the whole direct sum.
\end{enumerate}

\end{exs}

\begin{rem}
\begin{enumerate}
    \item     The above examples show that a classification of simple Hom-Lie algebras is not possible without additional restrictions on the twisting map or the multiplication law. 
    \item The twisting map of a simple Hom-Lie algebra needs not to be bijective.
\end{enumerate}
\end{rem}

\begin{defn}
A {\bf strongly simple Hom-Lie algebra} is a nonabelian anticommutative algebra where every twisting map endows it with a simple Hom-Lie structure.
\end{defn}

\begin{thm}\label{t4}
An anticommutative algebra is a strongly simple Hom-Lie algebra if and only if it is simple.
\end{thm}

\begin{proof}
For the zero twisting map, a strongly simple Hom-Lie algebra is just an anticommutative algebra that has no nontrivial proper ideal, and so it is simple. The converse is trivial.
\end{proof}

We shall denote by $\SSH$, $\MS$ and $\RS$, respectively, the class of strongly simple Hom-Lie algebras, the class of anticommutative algebras that have at least a multiplicative simple Hom-Lie structure, and the class of anticommutative algebras that have at least a regular simple Hom-Lie structure.
\begin{thm}
$\SSH=\MS$.
\end{thm}
\begin{proof}
Every algebra in $\SSH$ is multiplicative simple with respect to the zero map. On the other hand, by Lemma \ref{l41}, the twisting map of a multiplicative simple Hom-Lie algebra is either zero or regular. 
So, if $(A,\brk,\sigma)$ is a regular simple Hom-Lie algebra then, by the Structure Theorem \ref{t2}, $A=\bigoplus_{i=1}^n\s_i$ and for any nontrivial Hom-ideal $I$ there is $1\leq i\leq n$ such that $I_1=[I,\s_i]_{\sigma^{-1}}\neq0$, and let $I_{k+1}=[I_k,\s_i]_{\sigma^{-1}},~k\geq1$. So thanks to the new simplicity criterion, cf. Theorem \ref{t5}, there exists $k_0$ such that $I_{k_0}=\s_i\subset\sigma^{-k_0}(I)$, and since $\sigma$ permutes cyclically the simple components of $(A,\brk_{\sigma^{-1}})$ and by the $\sigma$-invariance of $I$, we get $I=A$. Hence the equality.
\end{proof}
\begin{cor}
  A multiplicative Hom-Lie algebra $(A,\left[\cdot,\cdot\right],\sigma)$ is simple, if it is not abelian and has no proper ideal.
\end{cor}

%%%%%%%%%%%%%%%%%%%%%%%%%%%%%%%%%%%%%%%%%%%%%%%%%%%%%%%%%%%%%%%%%%%%%%%%%%%%%

\section{Strongly* and Pure Strongly Simple Hom-Lie algebras}
In some cases strongly simple Hom-Lie algebras may have $\HS=0$, as it will be encountered later in Examples \ref{e5}. So it is natural to introduce new subclasses that avoid this phenomena.

\begin{defn}
We call {\bf $\text{strongly}^*$ simple Hom-Lie algebra} (resp. {\bf pure strongly simple Hom-Lie algebra}) a strongly simple Hom-Lie algebra $A$ that has $\HS(A)\neq 0$ (resp. 
if $\HS(A)\cap\operatorname{GL}(A)\neq0$).
We note respectively their classes $\Ss$ and $\PS$. 
\end{defn}

\begin{rem}
Obviously, $\RS\subset\PS\subset\Ss\subset\SSH=\MS$.
\end{rem}

\begin{thm}\label{t63}
In dimension $3$, every simple anticommutative algebra is isomorphic to the outside Yau's twist of the simple Lie algebra $\so(3,\F)$ with respect to some bijective linear map.
In addition, $\RS_3\varsubsetneq\PS_3=\Ss_3=\SSH_3=\MS_3$.
\end{thm}

\begin{proof}
Let $\{e_1,e_2,e_3\}$ be a basis of the $3$-dimensional simple anticommutative algebra $(A,[\cdot,\cdot])$ and set $\sigma$ be defined by $\sigma(e_3)=[e_1,e_2],~\sigma(e_1)=[e_2,e_3]$ and $\sigma(e_2)=[e_3,e_1]$. So by simplicity and the number of the brackets, $\sigma$ is bijective. Then $\cycl_{e_1,e_2,e_3}[\sigma(e_1),[e_2,e_3]]=\cycl_{1\leq  i\leq 3}[\sigma(e_i),\sigma(e_i)]=0$. Therefore $\sigma$ is a bijective twisting map, and so $(A,[\cdot,\cdot],\sigma)$ is a pure strongly simple Hom-Lie algebra. In addition, we remark that $(A,[\cdot,\cdot]_{\sigma^{-1}})\simeq\so(3,\F)$, i.e. $(A,\brk)$ is the outside Yau's twist of the Lie algebra $\so(3,\F)$ with respect to the bijective linear map $\sigma$.
\\For $\RS_3\varsubsetneq\PS_3$, we consider the simple Hom-Lie algebra $(A,\brk,\sigma)$ defined, as above, by :
$$[e_1,e_2]=\sigma(e_3),\,[e_2,e_3]=\sigma(e_1),\,[e_3,e_1]=\sigma(e_2)$$ 
where $\det(\sigma)\neq0$. For example, for $\sigma(e_1)=e_1,~\sigma(e_2)=e_2$ and $\sigma(e_3)\neq e_3$, we have $[\sigma(e_1),\sigma(e_2)]=[e_1,e_2]=\sigma(e_3)\neq\sigma^2(e_3)$, hence $\sigma$ cannot be multiplicative. 
\end{proof}

\begin{exs}\label{e5}
\begin{enumerate}
    \item Let $A_1$ be the $4$-dimensional anticommutative algebra defined by :
    $$[e_1,e_2]=e_3,\,[e_1,e_3]=e_4,\,[e_1,e_4]=e_1,\,[e_2,e_4]=e_2$$
    It's easy to show that $A_1$ is simple. Also, we find that a twisting map $\sigma$ has the form $\sigma(e_1)=a_{21} e_2+a_{31} e_3+a_{41} e_4$, $\sigma(e_4)=a_{44} e_4$ and $\sigma(e_2)=\sigma(e_3)=0$. Thus, $\HS(A_1)\neq0$ and for all $\sigma\in \HS(A_1),~\mathrm{det}(\sigma)=0$. Therefore $\PS\varsubsetneq\Ss$. 
 \item Let $A_2$ be the $4$-dimensional anticommutative algebra 
defined by the same relations as $A_1$, to which we add this new entry $:$ $[e_2,e_3]=e_1$. Clearly, $A_3$ is simple and we find that $\HS(A_2)=0$. So $\Ss\varsubsetneq\SSH$.
\end{enumerate}
\end{exs}

\begin{cor}
  $~~\RS\varsubsetneq\PS\varsubsetneq\Ss\varsubsetneq\SSH=\MS$.
\end{cor}

Finally, we should mention that a related problem to the class $\PS$ and to the above Theorem \ref{t63} was treated by Y. Frégier and A. Gohr in \cite{FG09} for the class of Hom-associative algebras. They proved that under a mild condition, of being left (resp. right or two-sided) weakly unital, a Hom-associative algebra $(\A,\mu,\alpha)$ with a bijective twisting map $\alpha$ is the Yau's twist of some associative algebra.
Where an algebra $(\A,\mu)$ is said to be Hom-associative with respect to the twisting map $\alpha$, if it satisfies the following Hom-associativity identity
\begin{align}
    \mu\circ(\mu\otimes\alpha)=\mu\circ(\alpha\otimes\mu)
\end{align}
and it is said to be left (resp. right, or two-sided) weakly unital, if 
\begin{align}
    \exists c\in \A,~\mu(c,x)=\alpha(x),~~\forall x\in \A
\end{align}
(resp. $ \exists c\in \A,~\mu(x,c)=\alpha(x),~~\forall x\in \A$, or $c$ is a left and right weak unity). \\
So that, if in addition $\alpha$ is bijective then $(\A,\alpha^{-1}\circ\mu)$ is an associative algebra with a left (resp. right or two-sided) unity element $c$, cf. \cite[Proposition 2.1]{FG09}.

%%%%%%%%%%%%%%%%%%%%%%%%%%%%%%%%%%%%%

\section*{Acknowledgments}
I thank professor Yves de Cornulier for kindly pointing me the result in the Theorem \ref{t55}, its key argument and the Remark \ref{r54}. Also, I am grateful to professor Manar Hayyani for the enlightening discussions and its useful comments.

% Your bilbigraphy           %<-------------------

\end{document}